\documentclass[10pt]{amsart}
\theoremstyle{plain}
\usepackage{amsmath, amssymb, amscd, graphicx, hyperref, mathtools, float}
\newtheorem{lemma}{Lemma}[section]
\newtheorem{theorem}[lemma]{Theorem}
\newtheorem{cor}[lemma]{Corollary}
\newtheorem{definition}[lemma] {Definition}

\newtheorem{conj}[lemma]{Conjecture}

\newtheorem*{theorem*}{Theorem}
\newtheorem*{conj*}{Conjecture}

\usepackage{color}

\newcommand{\sge}{\mathchoice
  {\mbox{\smaller$\ge$}}
  {\mbox{\smaller$\ge$}}
  {\mbox{\larger[-6]$\ge$}}
  {\mbox{\larger[-7]$\ge$}}
}

\newcommand{\sg}{\mathchoice
  {\mbox{\smaller$>$}}
  {\mbox{\smaller$>$}}
  {\mbox{\larger[-6]$>$}}
  {\mbox{\larger[-7]$>$}}
}

\newcommand{\lzero}{\mathchoice
  {\mbox{\smaller$0$}}
  {\mbox{\smaller$0$}}
  {\raisebox{-.7pt}{\larger[-3]$0$}}
  {\raisebox{-.7pt}{\larger[-4]$0$}}
}

\newcommand{\gezero}{\mathchoice
  {\mbox{\smaller$0$}}
  {\mbox{\smaller$0$}}
  {\raisebox{-.8pt}{\larger[-3]$0$}}
  {\raisebox{-.8pt}{\larger[-4]$0$}}
}

\newcommand{\Z}{{{\mathbb Z}}}
\newcommand{\lL}{\mathcal{L}}
\newcommand{\hfhat}{{{\widehat{HF}}}}
\newcommand{\Deltas}{\overline{\Delta}}
\newcommand{\dtge}{{{\mathcal{D}_{\mkern-11.2mu\phantom{a}^{\sge \mkern.3mu\gezero}\!}^{\tau}}}}
\newcommand{\dtgz}{{{\mathcal{D}_{\mkern-11.2mu\phantom{a}^{\sg \mkern.3mu\lzero}\!}^{\tau}}}}

\DeclarePairedDelimiter{\ceil}{\lceil}{\rceil}

\title{L-space fillings and Generalized solid tori}
\author{Thomas Gillespie}
\address{Department of Pure Mathematics and Mathematical Statistics, University of Cambridge, UK}
\email{tjg35@cam.ac.uk}
\thanks{}

\begin{document}
\begin{abstract}
Much work has been done recently towards trying to understand the topological significance of being an L-space. Building on work of Rasmussen and Rasmussen, we give a topological characterisation of Floer simple manifolds such that all non-longitudinal fillings are L-spaces. We use this to partially classify L-space twisted torus knots in \(S^1 \times S^2\) and resolve a question asked by Rasmussen and Rasmussen.
\end{abstract}
\maketitle
 
\section{Introduction}
If \(M\) is a rational homology \(S^3\), we can define the \emph{Heegaard Floer homology} \(\hfhat(M)\), which satisfies \(rk(\hfhat(M)) \geq |H_1(M)|\). It is natural to try and understand the manifolds with the simplest Heegaard Floer homology, those with \(rk(\hfhat(M)) = |H_1(M)|\), which we call \emph{L-spaces}. A conjecture of Boyer, Gordon and Watson \cite{lofg} suggests a concise description of such manifolds in terms of \(\pi_1(M)\).

\begin{conj}[\cite{lofg}]
An irreducible rational homology 3-sphere is an L-space if and only if its fundamental group is not left-orderable.
\end{conj}

A lot of recent work in the subject has gone into studying this conjecture. We follow a path taken by Rasmussen and Rasmussen in \cite{ras}
who approach the problem from the perspective of Dehn filling. More precisely, if \(Y\) is a rational homology \(S^1 \times D^2\), we will be interested in classifying the L-space fillings of \(Y\), i.e. the set
\[\lL(Y) = \{ \alpha \in Sl(Y) \, | \, Y(\alpha) \ \text{is an L-space}\}\]
where \(Sl(Y)\) is the set of filling slopes on \(\partial Y\). In particular, we will study the case \({\lL(Y) = Sl(Y) - \{l\}}\), where \(l\) is the homological longitude. 

In \cite{ras}, Rasmussen and Rasmussen define a subset \(\dtgz(Y) \subset H_1(Y)\) that (almost) completely determines the set \(\lL(Y)\). We strengthen a result of theirs about the behaviour of manifolds with \(\lL(Y) = Sl(Y) - \{l\}\).

\begin{theorem}
\label{thm:maintheorem}
The following are equivalent
\begin{enumerate}
\item  \(\lL(Y) = Sl(Y) - \{l\}\).
\item \(Y\) has genus \(0\) and has an L-space filling. 
\end{enumerate}
\end{theorem}

To prove this theorem, we show that for any \(Y\),  \(\dtgz(Y) = \emptyset\) if and only if \(Y\) is a \emph{generalized solid torus}.
This is closely related to the notion of a \emph{solid torus-like} manifold, introduced by Hanselman and Watson in \cite{watson}, which is a condition on the bordered Floer homology of the manifold. In particular, all solid torus-like manifolds are generalized solid tori, and if \(H_1(Y)\) is torsion-free the notions are equivalent.

This also allows us to simplify a result of Hanselman, Rasmussen, Rasmussen and Watson (\cite[Theorem 7]{hrrw}) concerning \emph{splicing} of manifolds. Given two manifolds with torus boundary, \(Y_1\) and \(Y_2\), and a map \(\varphi : \partial Y_1 \rightarrow \partial Y_2\), we can define the spliced manifold \(Y_\varphi\) by gluing \(Y_1\) and \(Y_2\) along their boundaries, by the map \(\varphi\).

\begin{cor}
Suppose \(Y_1\) and \(Y_2\) as above are Floer simple and boundary incompressible. Let \(\mathcal{L}^{\circ}_i\) be the interior of \(\mathcal{L}(Y_i) \subset Sl(Y_i)\). Then \(Y_\varphi\) is an L-space if and only if \(\varphi_*(\mathcal{L}^{\circ}_1) \cup \mathcal{L}^{\circ}_2 = Sl(Y_2)\). 
\end{cor}

As another application of this technology, we give a large class of twisted torus knots in \(S^1 \times S^2\) that have generalised solid torus complements. In \cite{vafaee}, Vafaee gives a class of twisted torus knots in \(S^3\) that are L-space knots, which is extended by Motegi in \cite{motegi}. We will show how these results can be applied to the study of twisted torus knots in \(S^1 \times S^2\) with Floer simple complements, and extend the classification.

\begin{theorem}
The twisted torus knot complement \(S^1 \times S^2 \setminus T(p,q;s,r)\) is a generalized solid torus if \((p,q) = 1\) and \(s \equiv \pm q \; (p)\).
\end{theorem}
We can use this to produce an set of examples to resolve a question asked in \cite{ras}.
\begin{cor}
There exists an infinite set of distinct Floer simple manifolds with the same Turaev torsion.
\end{cor}
While studying this problem we also come across an interesting class of two component link complements.

\begin{theorem}
\label{thm:inffam}
There exists an infinite family of manifolds \(Z_{p,q}\) with \(\partial Z_{p,q} = T^2 \cup T^2\) such that any filling \(Z_{p,q}(\alpha, \beta)\) with \(b_1(Z_{p,q}(\alpha, \beta)) = 0\) is an L-space.
\end{theorem}

\subsection*{Acknowledgements:} The author would like to thank his supervisor Jacob Rasmussen for introducing this problem, and for support and helpful 
conversations throughout.

\section{Generalized solid tori}

First, we will fix some conventions and notation. Throughout, \(Y\) will be a rational homology \(S^1 \times D^2\). Let \(\lambda\) be the homological longitude and \(\mu\) a meridian of \(\partial Y\). Write \(H_1(Y) = \mathbb{Z} \oplus T\), \(i : H_1(\partial Y) \to H_1(Y)\) for the induced inclusion map. Then 
\[
{Im(i) =  g_Y \mathbb{Z} \oplus \mathbb{Z}/g_Y\mathbb{Z} = g_Y \mathbb{Z} \oplus T' \subset H_1(Y)}
\]
where \(i(\mu)\) generates \(g_Y \mathbb{Z}\) and \(\sigma := i(\lambda)\) generates \(T'\).

Fix a projection of the free part \(\phi:H_1(Y) \to \Z\) and let \(t\) be a generator of \(H_1(Y) / T\). Finally, let \(\iota : H_1(Y) \to H_1(Y(\lambda)) = \mathbb{Z} \oplus T^* \cong \mathbb{Z} \oplus T/T'\) and \(k_Y = |T^*|\).

Let \(\tau(Y)\) be the Turaev torsion, which is an element of \((1-t)^{-1}\mathbb{Z}[H_1(Y)] \subset Q(\mathbb{Z}[H_1(Y)]\), where \(t\) generates the free part of \(H_1(Y)\).
By writing \((1-t)^{-1} = \sum_{i=0}^\infty t^i\), we can consider \(\tau(Y)\) as an element of the Novikov ring
\[
 \Lambda_\phi[H_1(Y)] = \Big\{ \sum_{h \in H_1(Y)} a_h [h] \ \Big| \ \#\{h \, | \, a_h \neq 0, \phi(h)<k\}<\infty \ \text{for all} \ k\Big\}.
\]
i.e. as a formal sum of elements in \(H_1(Y)\). We will normalise \(\tau(Y)\) so that \(a_h = 0\) for \(\varphi(H) < 0\) and \(a_0 \neq 0\).
With this representation, we can define \({S[\tau(Y)] = \{h \in H_1(Y) \, | \, a_h\neq 0\}}\), the {\em support} of \(\tau(Y)\).
We say that \(Y\) is \emph{Floer simple} if \(\lL(Y)\) is an interval inside \(Sl(Y)\).

Finally, we define an analogue of the Alexander polynomial. Let \(\Phi :  \Lambda_\phi[H_1(Y)] \rightarrow \mathbb{Z}[t, t{-1}]\) be the map induce from \(\varphi\), then set \(\overline{\Delta}(Y) = (1-t) \Phi(\tau(Y))\).

\begin{definition}[{\cite[Definition 1.5]{ras}}]
If \(Y\) is a Floer simple manifold, we define 
\[ \dtge(Y) = \{x-y\,|\, x \notin S[\tau(Y)], y \in S[\tau(Y)],\phi(x)\geq \phi(y)\} \cap \mathrm{im} \, \iota \subset H_1(Y)\]
and write 
\(\dtgz(Y)\) for the subset of \(\dtge(Y)\) consisting of those elements with \(\phi(h)>0\). 
\end{definition}

\begin{figure}
\begin{center}
  \begin{tabular}{ | l | l | }
    \hline
    \(Y\) & a rational homology \(S^1 \times D^2\) \\ \hline
    \(\lambda\) & a homological longitude of \(\partial Y\) \\ \hline
    \(\mu\) & a meridian of \(\partial Y\) \\ \hline
    \(T\) & the torsion part of \(H_1(Y)\) \\ \hline
    \(\phi:H_1(Y) \to \Z\) & a (fixed) projection of the free part of \(H_1(Y)\) \\ \hline
    \(t\) & a generator of the free part of \(H_1(Y)\) \\ \hline
    \(i : H_1(\partial Y) \to H_1(Y)\) & the induced inclusion map \\ \hline
    \(\sigma := i(\lambda)\) & the inclusion of the longitude\\ \hline
    \(g_Y := ord(\sigma)\) & \\ \hline
    \(\iota : H_1(Y) \to H_1(Y(\lambda))\) & the induced inclusion map from the \(\lambda\) filling \\ \hline
    \(T^*\) & the torsion part of \(H_1(Y(\lambda))\) \\ \hline
    \(k_Y := |T^*|\) & \\ \hline
    \(\tau(Y) \in \Lambda_\phi[H_1(Y)]\) & the Turaev torsion \\ \hline
    \(\overline{\Delta}(Y)\) & an analogue of the Alexander polynomial \\ \hline
  \end{tabular}
\end{center}
\caption{Conventions and notation}
\end{figure}

The significance of this definition is demonstrated by \cite[Theorem 1.6]{ras}, which says that if \(Y\) is Floer simple then either \(i^{-1}(\dtgz(Y))\) is empty, in which case \(\lL(Y) = Sl(Y) - \{l\}\), or it is nonempty and \(\lL(Y)\) is a closed interval whose endpoints are consecutive elements of \(\dtgz(Y)\). I.e. the set \(\dtgz(Y)\) (and hence the Turaev torsion) along with one L-space filling, completely determines \(\lL(Y)\).

Finally, we need to define the notion of a \emph{generalized solid torus}.

\begin{definition}
 A {\em generalized solid torus} is a Floer simple manifold \(Y\) with \(\deg \Deltas(Y) < g_Y\). 
\end{definition}

This definition is motivated by the fact that if \(Y\) is Floer simple and \(\deg \Deltas(Y) < g_Y\), then \(\dtgz(Y) = \emptyset\) (\cite[Proposition 7.1]{ras}). This can be compared to the notion of solid torus-like manifolds introduced by Hanselman and Watson in \cite{watson}, which we will define and discuss later.

Clearly all solid torus-like manifold are also generalised solid tori. We will show later that they also must have \(g_Y = 1\), which (combined with our main theorem) greatly restricts their form.

We are now in a position to prove our main theorem. The primary improvement over the proof of Rasmussen and Rasmussen in \cite{ras} is a more careful treatment of the Turaev torsion, using lemma \ref{thm:torsionlemma}, which allows us to more precisely control the behaviour of the polynomials \(q_{i,s}\) defined below.

\begin{theorem}
\label{4equivs}
Suppose that \(Y\) is a Floer simple rational homology \(S^1 \times D^2\) and \(\mathcal{D}_{> 0}^{\tau} = \emptyset\). Then \(Y\) is a generalized solid torus.
\end{theorem}
\begin{proof}

Let \(\Sigma_{Y} = \sum_{s \in T} s\) and \(\Sigma_{Y(\lambda)} = \sum_{s \in T^*} s\). We will need some notation and two lemmas from \cite{turaev}.

\begin{lemma} \cite[II.4.5]{turaev} \label{thm:polynomialpart}
When \(b_1(Y) = 1\) we can define the \emph{polynomial part} of the torsion
\[ [\tau](Y) = 
\begin{cases}
\tau(Y) + (t-1)^{-1} Q_Y(t) \Sigma_{Y} & \text{ if } \partial Y = S^1 \times S^2 \\
\tau(Y) + (t-1)^{-1} (t^{-1}-1)^{-1} Q_Y(t) \Sigma_{Y} & \text{ if } \partial Y = \emptyset 
\end{cases}
 \]
for some polynomial \(Q_Y(t)\). Then \([\tau](Y) \in (\frac{1}{2} \mathbb{Z})[H_1(Y)]\), i.e. \([\tau](Y)\) is a polynomial.
\end{lemma}

\begin{lemma}\cite[VII.1.5]{turaev} \label{thm:torsionlemma}
With \(Y\) as above,
\[\iota([\tau(Y)]) = \pm (t^{g_Y} - 1)[\tau](Y(\lambda)) + \Sigma_{Y(\lambda)} P_Y(t) \]
for some polynomial \(P_Y(t)\).
\end{lemma}

Write

\[\tau(Y) = \sum_{s \in T^*} s \sum_{i=0}^{\infty} q_{i, s}(\sigma) t^i \]

where, since \(Y\) is Floer simple, all of the coefficients of \(q_{i, s}(\sigma)\) are 1 or 0. Then we can use Turaev's lemma to show that

\begin{align*}
\iota(\tau(Y)) &=  \sum_{s \in T^*} s \sum_{i=0}^{\infty} q_{i, s}(1) t^i  \\
&= \iota([\tau](Y)) + g_Y \Sigma_{Y(\lambda)} \sum_{i=C}^{\infty} t^i \\
&= \pm (t^{g_Y} - 1)[\tau](Y(\lambda)) + \Sigma_{Y(\lambda)} (P_Y(t) +  g_Y \sum_{i=C}^{\infty} t^i)
\end{align*}

So, by taking coefficients of the \(s \in T^*\)

\[(t^{g_Y} - 1) \; \mid \; Q_s(t) := \sum_{i=0}^{\infty} q_{i, s}(1) t^i - P_Y(t) -   g_Y \sum_{i=C}^{\infty} t^i\]

Since \(\mathcal{D}_{>0}^\tau = \emptyset\), if the coefficient (in \(\tau\)) of some \(s \sigma^a t^k\) is 1, then the coefficient of \(s \sigma^{a'} t^{k + n g_Y}\) must also be 1 for all \(a'\) and \(k > 0\). So there is at most one value of \(n\) such that the polynomial \(q_{k + n g_Y, s}(\sigma)\) is not either 0 or \(\Sigma_{Y(\lambda)}\). 
 Therefore, the subsequence \(q_{k + n g_Y, s}(1)\) must look like \(0, 0, ..., 0, q_{k + n_{k, s} g_Y, s}(1), g_Y, g_Y, ...\)

We now cut up \(Q_s(t)\) into groups of powers of \(t\) that correspond to equivalence classes in \(\mathbb{Z}/g_Y \mathbb{Z}\), i.e.

\[Q_{s,k}(t) := \sum_{n=0}^{\infty} q_{k + n g_Y, s}(1)  t^{k + n g_Y} - P_{Y,k}(t) -  g_Y  \sum_{n = C_k}^{\infty} t^{k +  n g_Y}\]

where \(C_k = \ceil{(C - k)/g_Y}\) and \(P_{Y,k}(t)\) contains the terms within \(P_Y(t)\) of the form \(c t^{k + n g_Y}\). Clearly this process preserves divisibility by \(t^{g_Y} - 1\), and so we also have divisibility of the following differences (wlog \(n_{k, s'} \geq n_{k, s}\))

\begin{equation*}
\begin{split}
(t^{g_Y} &- 1) \; \mid \; (Q_{s, k}(t) - Q_{s', k}(t)) \\
&= q_{k+n_{k,s} g_Y, s}(1) t^{k+n_{k,s} g_Y} + g_Y \sum_{n = n_{k, s} + 1}^{n_{k, s'} - 1} t^{k +  ng} + (g_Y - q_{k+n_{k,s'} g_Y, s'}(1)) t^{k+n_{k,s'} g_Y}
\end{split}
\end{equation*}

from which we can conclude that \(n_{k, s'} = n_{k, s}\) and \(q_{k + n g_Y, s}(1) = q_{k + n g_Y, s'}(1)\) for all \(s, s'\). So if we write \(a_i = q_{i, s}(1)\), we get a simple form for the \emph{Milnor torsion}

\[\bar{\tau}(Y) = pr( \sum_{s \in T^*} s \sum_{i=0}^{\infty} q_{i, s}(\sigma) t^i ) = k_Y \sum_{i=0}^{\infty} a_i t^i \]

where \(pr : \mathbb{Z}[H_1(Y)] \to \mathbb{Z}[H_1(Y)/T]\). 

The rest of the proof proceeds analogously to \cite{ras}, but we outline it here for clarity and completeness.

\begin{lemma}
There is a constant c so that \(\sum\limits_{i \equiv k \; (g_Y)} k_Y a_i \equiv k_Y(k + c) \; (g_Y)\).
\end{lemma}
\begin{proof}
By \cite[Lemma 7.3]{ras}, \(p(\bar{\Delta}(t)) = p((t-1)\bar{\tau}(Y)) \sim k_Y (1 - t^{g_Y}) / (1 - t)\) where \(p : \mathbb{Z}[t] \to \mathbb{Z}[t] / (t^{g_Y} - 1)\), so \(\bar{\Delta}(t)\)
can be obtained from \(k_Y (1 - t^{g_Y})\) by moves of the form \(f(t) \to f(t) + t^i - t^{g_Y-i}\) and \(f(t) \to t^c f(t)\). It is then sufficient to note that if \(f(t)/(1-t) = \sum_{i=0}^{\infty} a_i t^i\)
satisfies the conditions of the lemma, so do \(w f(t) + t^i - t^{g_Y-i}\) and \(t^c f(t)\).
\end{proof}

This implies that the subsequence (\(k_Y a_{k + n g_Y}\)) has the form 
\[0, 0, ..., 0, k_Y k, k_Y g_Y, k_Y g_Y, ....\]

so \(\bar{\tau}(Y)\) is obtained from 

\[\bar{\tau}_0 = k_Y(t + 2t^2 + ... + (g_Y - 1)t^{g_Y-1} + g_Y t^{g_Y} + g_Y t^{g_Y + 1} + ...) = k_Y \frac{t^{g_Y} - 1}{(t-1)^2}\]

by elementary shifts of the form \(k_Y(a t^i + (g_Y - a)t^{i+g_Y})\).

Finally, consider the map \(F(Q(t)) = p((1-t)Q(t))\),  \(p : \mathbb{Z}[t] \to \mathbb{Z}[t] / (t^{g_Y} - 1)\). One can check that if \(f(t) - g(t) = k_Y(a t^i + (g_Y - a)t^{i+g_Y})\), \(F(f) - F(g) = g_Y k_Y(t^i - t^{i+1})\).
But since \(\bar{\tau}(Y)\) is obtained from \(\bar{\tau}_0\) by these shifts, and since \(F(\bar{\tau_0}) = k_Y \frac{t^{g_Y} - 1}{(t-1)} = F(\bar{\tau}(Y))\), all residue classes must be shifted by the same amount, i.e. \(\bar{\tau}(Y) \sim \bar{\tau}_0\). Hence \(Y\) is a generalised solid torus.

\end{proof}

We have the following corollary, which follows immediately as an extension of \cite[Proposition 1.9]{ras} (and which contains Theorem \ref{thm:maintheorem}).

\begin{cor}
The following are equivalent
\begin{enumerate}
\item  \(\lL(Y) = Sl(Y) - \{l\}\).
\item \(Y\) is Floer simple and \(\dtgz(Y) = \emptyset\).
\item \(Y\) is Floer simple and has genus \(0\). 
\item \(Y\) has genus \(0\) and has an L-space filling.
\end{enumerate}
\end{cor}

This immediately gives us a result about solid torus-like manifolds, which we define now.

\begin{definition}[{\cite[Definition 3.23]{watson}}]
A loop-type manifold $M$ is solid torus like if it is a rational homology solid torus and every loop in the representation of $\widehat{CFD}(M,\alpha,\beta)$ is solid torus-like.  Equivalently, \(\dtge(Y) = \emptyset\).
\end{definition}

\begin{cor}
\label{cor:stlIsBc}
If \(Y\) is solid torus-like, then it is boundary compressible.
\end{cor}
\begin{proof}
A solid torus-like manifold has \(\dtge(Y) = \emptyset\), and so \(\dtgz(Y) = \emptyset\). Hence there is a generator \(\Sigma\) of \(H_2(Y, \partial Y)\) with genus 0.

Also, since \(\dtge(Y) = \emptyset\), we must have that \(q_{i, s}(\sigma) = \sum_{i=0}^{g_Y} \sigma^i\) for all \(i\) and \(s\), and so \(\bar{\Delta}(t) = k_Y g_Y \sum_{i=0}^\infty t^i\). Therefore \(g_Y = 1\) by \cite[Lemma 7.3]{ras}.
So \(\Sigma\) is a compressing disk for \(Y\).
\end{proof}

We will now turn our attention to an application to \emph{splicing} of manifolds. Suppose that \(Y_1\) and \(Y_2\) are rational homology solid tori, \(\varphi : \partial Y_1 \to \partial Y_2\) is an orientation reversing diffeomorphism.
We say that \(Y_\varphi = Y_1 \cup_\varphi Y_2\) is obtained by splicing \(Y_1\) and \(Y_2\) together along \(\varphi\).

\begin{cor}
Suppose \(Y_1\) and \(Y_2\) as above are Floer simple and boundary incompressible. Let \(\mathcal{L}^{\circ}_i\) be the interior of \(\mathcal{L}(Y_i) \subset Sl(Y_i)\). Then \(Y_\varphi\) is an L-space if and only if \(\varphi_*(\mathcal{L}^{\circ}_1) \cup \mathcal{L}^{\circ}_2 = Sl(Y_2)\). 
\end{cor}
\begin{proof}
By \ref{cor:stlIsBc} the \(Y_i\) are not solid torus-like, and we can apply \cite[Theorem 7]{hrrw}.
\end{proof}

\section{Twisted Torus Knots}
In an effort to find examples of generalised solid tori, we now turn our attention to twisted torus knots. We write \(T(p,q;s,r)\) for the knot in figure \ref{fig:ttk}, considered as a knot in \(S^1 \times S^2\).

\begin{figure}[ht]
    \centering
    \def\svgwidth{\columnwidth}
    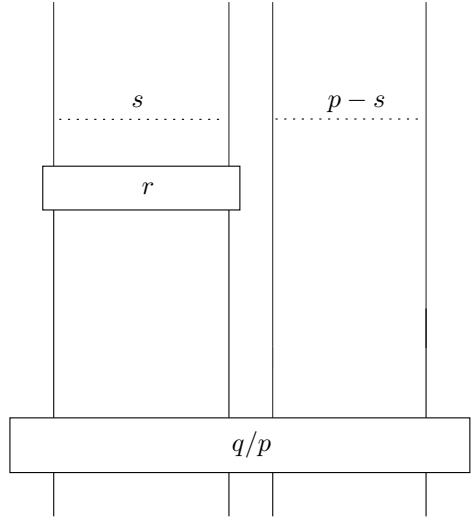
    \caption{The twisted torus knot \(T(p,q;s,r)\)}
    \label{fig:ttk}
\end{figure}

In \cite{vafaee}, Vafaee studies twisted torus knots \(K(p,q;s,r) \subset S^3\), and classifies as subset that have L-space fillings proving

\begin{theorem}[{\cite[Theorem 1]{vafaee}}]
For \( p \geq 2, k \geq 1, r > 0\) and \(0 < s < p\), the twisted torus knot, \(K(p, k p \pm 1; s, r)\), is an L-space knot if and only if either \(s = p - 1\) or \(s \in \{2, p-2\}\) and \(r = 1\).
\end{theorem}
This theorem can be used to prove the existence class of generalised solid tori as complements of knots in \(S^1 \times S^2\).
\begin{cor}
\label{thm:vafaeeExtension}
For \( p \geq 2, k \geq 1, r > 0\) and \(0 < s < p\), the complement of the twisted torus knot, \(T(p, k p \pm 1; s, r) \subset S^1 \times S^2\), is a generalised solid torus if either \(s \in \{1, p - 1\}\) or \(s \in \{2, p-2\}\) and \(r = 1\).
\end{cor}
\begin{figure}[ht]
    \centering
    \def\svgwidth{\columnwidth}
    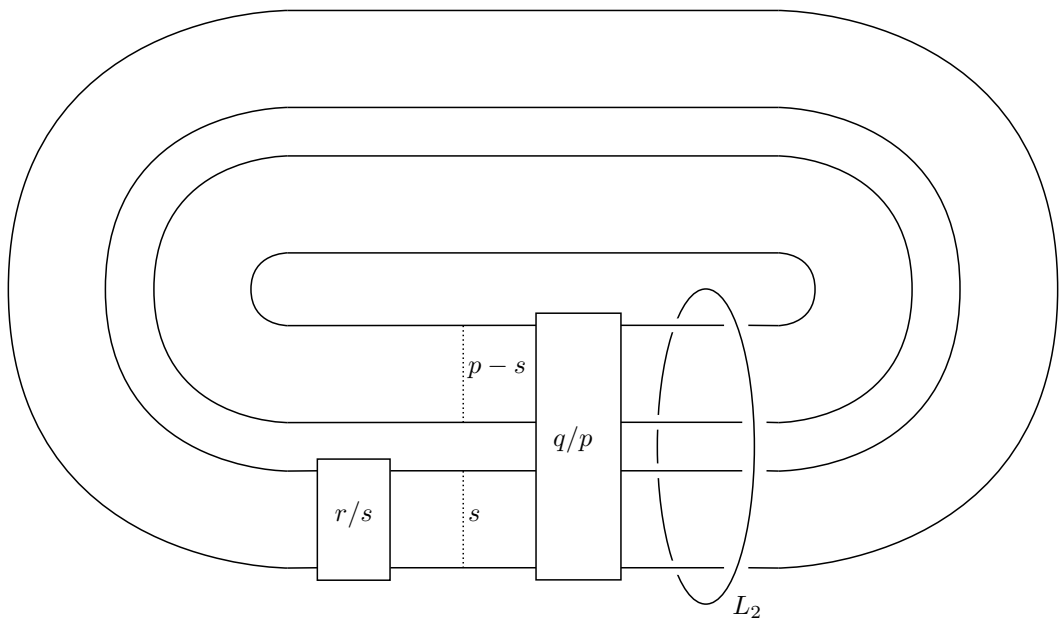
    \caption{The link \(L\)}
    \label{fig:torusKnot}
\end{figure}
We could similarly extend the results of Motegi to twisted torus knots in \(S^1 \times S^2\). Instead, we will give a more direct geometric proof.
\begin{theorem}
\label{thm:ttkAreGST}
The twisted torus knot complement \(S^1 \times S^2 \setminus T(p,q;s,r)\) is a generalized solid torus if \((p,q) = 1\) and \(s \equiv \pm q \; (p)\).
\end{theorem}

To prove this theorem, we require a lemma concerning the existence of a surface in the knot complement.

\begin{lemma}
\label{thm:sphere}
Consider the black tangle \(L \subset I \times D^2\) in figure \ref{fig:ttk1}.

\begin{figure}[ht]
    \centering
    \def\svgwidth{\columnwidth}
    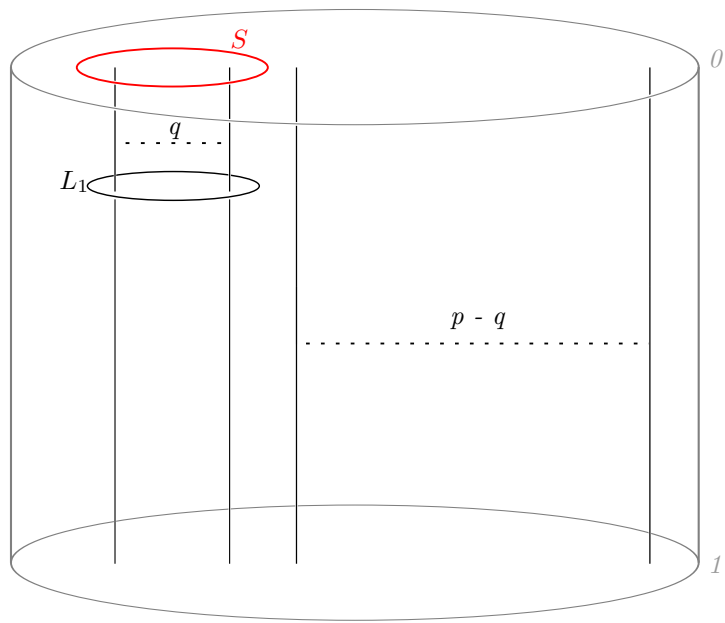
    \caption{A link in \(I \times D^2\). Under the \(q/p\) gluing, \(S_0\) glues to \(S_1\)}
    \label{fig:ttk1}
\end{figure}

There exists an embedded surface \(\Sigma \subset I \times D^2 - \nu(L)\) satisfying the following properties:

\begin{enumerate}
\item The boundary of \(\Sigma\) consists of \(q\) parallel copies of \(\{\frac{1}{2}\} \times \partial D^2\), \(p\) copies of the longitude of \(L_1\) and \(B \subset \{0,1\} \times D^2\).
\item If we glue \(\{0\} \times D^2\) to \(\{1\} \times D^2\) by an automorphism of \(D^2\) that shifts each strand of \(L\) left \(q\) times (i.e. the gluing that will close \(L_2\) to form a \(p/q\) torus knot), the boundary components in \(B\) 
join up to each other to form a surface \(\Sigma'\) with boundary only on \(L_1\) and \(I \times \partial D^2\). Moreover, \(\Sigma'\) has genus 0. 
\item \(\Sigma \cap S \times [0,1/2] = \emptyset\)
\end{enumerate}
\end{lemma}

\begin{proof}
We will build \(\Sigma\) inductively. Let \(w = kp + q\), and consider the black tangle \(L \subset I \times D^2\) in figure \ref{fig:ttk2}.

\begin{figure}[ht]
    \centering
    \def\svgwidth{\columnwidth}
    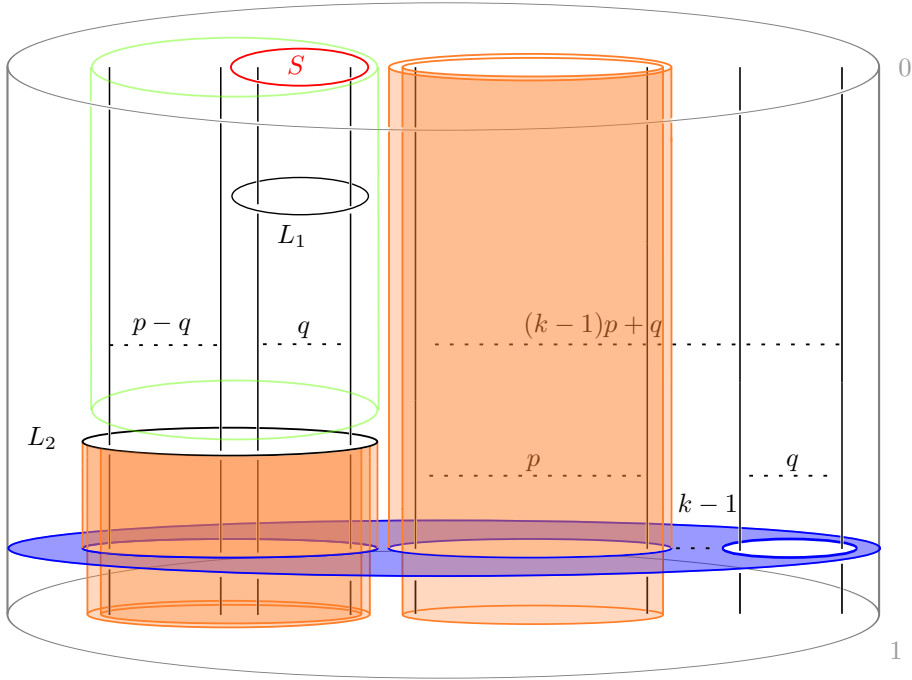
\caption{A link and an embedded surface in \(I \times D^2\). The green cylinder contains a reflected copy of figure \ref{fig:ttk1}. The pink pipes run through the holes in the blue surface.}
\label{fig:ttk2}
\end{figure}

Then by induction we can find a surface in a neighbourhood of the first \(p\) strands and \(L_1\), whose boundary components are \(p\) copies of the longitude of \(L_1\), \(q\) copies of the longitude of \(L_2\) and a 
boundary component \(B \subset \{0,1\} \times D^2\). Call this surface \(\Sigma_0\) and write \(B = B_0 \cup B_1\) where \(B_i \subset {i} \times D^2\). Also note that \(\Sigma_0\) is `surrounded' by \(L_2\).

Next, take \(k-1\) copies of \(B_1 \times I\) and insert them along the \(k-1\) remaining groups of \(p\) strands, i.e. the first one along the strands \(p+1, ..., 2p\), the second along \(2p+1, ..., 3p\) etc. Along the final \(q\) strands of the braid, we paste a copy of \(B' \times I\), where \(B'\) consists of the components of \(B_0\) contained inside the circle \(S\). Take the union of all of the pieces that we have defined so far, and call it \(\Sigma_1\). Notice that at this stage, the \(p/w\) twisted gluing will correctly match up all boundary components of \(\Sigma_1\) inside \(\{0,1\} \times D^2\) (since \(w \equiv q \; (p)\)).

Next we take \(p\) parallel copies of the blue surface in figure \ref{fig:ttk2}. As shown in the picture, each blue surface has one boundary component on \( I \times \partial D^2\), \(k\) boundary components \(C_i\) containing \(p\) strands, and one boundary component \(E\) containing \(q\) strands.

For \(C_1\), we run a pipe along the outside of \(\Sigma_1\) to join \(C_1\) to \(L_2\), and \(k-1\) parallel copies of pipes from \(\{1\} \times D^2\) to \(L_2\) each running inside the previous pipe, but still outside of \(\Sigma_1\). An example is shown in figure \ref{fig:ttk2}, shown in pink. Similarly for each other \(C_i\), run a pipe from \(C_i\) to \(\{0\} \times D^2\), and \(k+1-i\) pipes inside that one from \(\{1\} \times D^2\) to \(\{0\} \times D^2\), all of them running outside of \(\Sigma_1\). Repeat this process for each parallel copy of the blue surface. Call this new surface \(\Sigma_2\), and notice that the boundary components still all match up under a \(p/w\) twisted gluing. Also notice that the boundary of \(\Sigma_2\) contains \(p\) parallel copies of \(\{\frac{1}{2}\} \times \partial D^2\) and \(w = q + k p\) copies of the longitude of \(L_2\).

Finally we run a pipe down from each copy of the final blue boundary component \(E\), and \(p\) parallel pipes joining the red circle \(S\) to \(L_1\). We know by condition 3 that these pipes don't intersect \(\Sigma_2\). These \(p\) boundary components on \(L_1\) can be paired up with the \(p\) existing boundary components on \(L_1\), and pushed off. Finally we do an \(\infty\) filling on \(L_1\) to get our surface \(\Sigma\).

To check the genus, write \(\Sigma_1 '\) for the gluing of \(\Sigma_1\), and \(\Sigma'\) for the gluing of our new surface \(\Sigma\). Note that \(\Sigma_1 '\) is homeomorphic to the surface that we would get by gluing \(\Sigma_0\) inside figure \ref{fig:ttk1}, and so has genus 0.
To obtain \(\Sigma_2 '\), we just added \(p\) copies of the blue surface (and some pipes) all of which are disjoint from \(\Sigma_1 '\), and so \(\Sigma_2 ' \cong  \Sigma_1 ' \coprod p D^2_k \), where \(D^2_k\) is a \(k\)-punctured disc. So \(\Sigma_2'\) consists of \(p+1\) disjoint pieces all of which have genus 0. The final step simply glues one boundary component of each \(D^2_k\) to a boundary component of \(\Sigma_1 '\), so clearly \(\Sigma'\) has genus 0.

Note finally that condition 3 is clearly satisfied by the construction.
\end{proof}

\begin{lemma}
Let \(r_i \in \mathbb{Q \cup \{\infty\}}\), \(r_1 \neq 0\) and \(r_3 \neq \infty\). Then the filling \(Y_{r_1,0, r_3}\) of the following link \(L \subset S^3\) (in figure \ref{fig:ttk3}) is an L-space.

\begin{figure}[ht]
    \centering
    \def\svgwidth{\columnwidth}
    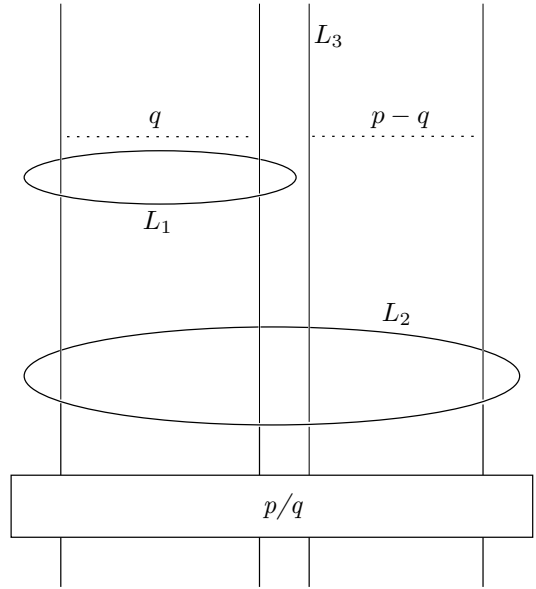
\caption{A link in \(S^3\) (the top and bottom should be identified)}
\label{fig:ttk3}
\end{figure}
\end{lemma}
\begin{proof}
Note that \(Y_{-,0,r_3}\) has genus 0, since we can take the surface \(\Sigma\) from lemma \ref{thm:sphere} and cap off the boundary components on \(L_2\). Also, \(Y_{\infty,0,r_3}\) is an L-space since it is the \(r_3\) filling of \(T(p,q) \subset S^1 \times S^2\), which is a generalized solid torus by \ref{thm:vafaeeExtension}.
So by theorem \ref{thm:maintheorem}, \(\lL(Y_{-,0,r_3}) = Sl(Y_{-,0,r_3}) - \{l\}\), hence \(Y_{r_1,0,r_3}\) is an L-space.

\end{proof}

\begin{proof}[Proof of theorem \ref{thm:ttkAreGST}]
\(S^1 \times S^2 \setminus T(p,q;q,r) = S^1 \times S^2 \setminus T(p,q;q',r) = Y_{1/r,0,-}\), so by the above lemma all of the non-longitudinal fillings are L-spaces. Hence \(S^1 \times S^2 \setminus T(p,q;s,r)\) is a generalized solid torus.
\end{proof}

Theorem \ref{thm:inffam} then follows easily by considering the family of manifolds \(Y_{\_,0,\_}\) above, for all \(p\) and \(q\).

This class of manifolds allows us to resolve a question of Rasmussen and Rasmussen. In \cite{ras}, Rasmussen and Rasmussen asked if there are infinitely many distinct Floer simple manifolds with the same Turaev torsion. We can build such a family as follows. Consider the manifold \(Y_{\_,0,\_}\) as above, with \(p = 5\) and \(q = s = 2\). Using SnapPy (\cite{snappy}) we can check that this manifold is hyperbolic, and apply Thurston's hyperbolic Dehn filling theorem.

\begin{theorem}[Thurston's hyperbolic Dehn filling theorem (\cite{thurston}]
Suppose \(M\) is hyperbolic and \(M(p/q)\) is the filling of one of the boundary components of \(M\). Then
\begin{itemize}
\item \(M(p/q)\) is hyperbolic for all but finitely many \(p/q\)
\item \(vol(M(p/q)) < vol(M)\) when \(M(p/q)\) is hyperbolic
\item \(vol(M(p/q)) \rightarrow vol(M)\) as \(p^2 + q^2 \rightarrow \infty\)
\end{itemize}
\end{theorem}

Therefore there is some subsequence \(Y_{1/r_i, 0, \_}\) that are distinct (since they have different hyperbolic volumes) but have the same \(\overline{\Delta}\). Since, for a Floer simple manifold, there are only finitely many possible torsions corresponding to a choice of \(\overline{\Delta}\), there must therefore be an infinite set of distinct Floer simple manifolds with the same torsion.

\bibliographystyle{plain}
\bibliography{references}

\end{document}